\documentclass[12pt]{amsart}
\usepackage{amssymb,amscd,color}
\usepackage{xypic}
\usepackage{enumitem}

\usepackage[
urlcolor=blue,
colorlinks=true,
linkcolor=blue,
citecolor=blue,
]{hyperref}

\newtheorem{thm}{Theorem}[section]
 
\newtheorem{lem}[thm]{Lemma}

\newtheorem{prop}[thm]{Proposition}
\newtheorem{cor}[thm]{Corollary}
\theoremstyle{definition}
\newtheorem{defn}[thm]{Definition}

\newtheorem{rmk}[thm]{Remark}

\theoremstyle{remark}

\numberwithin{equation}{section}

\newcommand{\alb}{\op{Alb}}

\newcommand{\op}{\mathrm}

\newcommand{\mf}{\mathcal{F}}
\newcommand{\mo}{\mathcal{O}}

\newcommand{\Z}{\mathbb{Z}}
\newcommand{\B}{\mathbb{B}}
\newcommand{\C}{\mathbb{C}}

\newcommand{\aut}{\mathrm{Aut}}
\newcommand{\can}{\mathrm{can}}
\newcommand{\diff}{\mathrm{Diff}}
\newcommand{\fix}{\mathrm{Fix}}

\newcommand{\id}{\mathrm{id}}

\newcommand{\sing}{\mathrm{sing}}

\newcommand{\restr}[1]{{\raisebox{-0.0\height}{$\mid_{#1}$}}}

\begin{document}
\title{{Surfaces of general type with $\MakeLowercase{q}=2$ are rigidified}}
\author{Wenfei Liu}

\dedicatory{Dedicated to Professor Fabrizio Catanese on the occasion of his 65th birthday}
\address{Wenfei Liu \\School of Mathematical Sciences\\ Xiamen University\\Siming South Road 422\\361005 Xiamen\\ Fujian\\ P. R. China}
\email{wliu@xmu.edu.cn}
\thanks{Several ideas used here emerge during the many helpful discussions with Jin-Xing Cai.  My research benefited a lot from the support of Leibniz Universit\"at Hannover and of Peking University while I was working in these universities. I was partially supported by NSFC (No.~11425101, No.~11471020) and by the Recruitment Program for Young Professionals. The author is grateful to Yi Liu for answering a question in topology.}
\date{}

\begin{abstract}
Let $S$ be a minimal smooth projective surface of general type with irregularity $q=2$. We show that, if $S$ has a nontrivial holomorphic automorphism acting trivially on the cohomology with rational coefficients, then it is a surface isogenous to a product. As a consequence of this geometric characterization, one infers that no nontrivial automorphism of surfaces of general type with $q=2$ (which are not necessarily minimal) can be homotopic to the identity. In particular, such surfaces are rigidified in the sense of Fabrizio Catanese.\end{abstract}
\maketitle

\section*{Introduction}
It is an interesting phenomenon in algebraic and complex geometry that sometimes topology determines geometry. For example, in the study of (biholomorphic) automorphisms of compact complex manifolds, one can ask when two homotopic automorphisms are in fact identical. Obviously, the answer is no if the automorphism group has a positive dimension; algebraic curves of genus at most 1 are such examples. On the other hand, when the manifold can be endowed with a K\"ahler metric with nonpositive sectional curvature, one obtains the rigidity of automorphisms by resorting to the uniqueness of harmonic maps within a homotopy class into such manifolds (see Theorem~\ref{thm: harmonic}). Algebraic curves of genus at least 2 and the products thereof fall into this category.

In general, a compact complex manifold neither has a positive dimensional automorphism group nor admits a K\"ahler metric with nonpositive curvature. So other methods are needed to investigate the rigidity of automorphisms. 

As is usual in algebraic topology, one can consider the induced action of an automorphism on the cohomolgy.  An automorphism of a compact complex manifold is  \emph{numerically trivial} if it acts trivially on the cohomology groups with rational coefficients (or with complex coefficients); it is called \emph{cohomologically trivial} if it acts trivially on the cohomology groups with $\Z$-coefficients. Note that, in comparing two automorphisms, we easily reduce to the case where one of the automorphism is the identity map. It is evident that automorphisms homotopic to the identity is cohomologically trivial, and the cohomologically trivial automorphisms are numerically trivial. 

Apart from giving a topological characterization of the identity map, a faithful action of the automorphisms on the cohomology groups is useful in constructing fine moduli spaces of algebraic varieties (see \cite{Pop77}).

It is well-known that there are no numerically trivial automorphisms on curves of genus $\geq 2$, while for curves of genus $\leq 1$ the numerically trivial automorphisms are exactly those lying in the identity component of the automorphism group.  Much attention has been paid to the automorphisms of K3 surfaces and their analogues (\cite{BR75,MN84, Muk10,Bea83, BNS11, Ogu12, MW17}). Numerically trivial automorphisms of elliptic surfaces and of surfaces of general type have been investigated  by Peters (\cite{Pet79, Pet80}) and by Cai and his coauthors (\cite{Cai04, Cai06a, Cai06b, Cai07, Cai09, Cai10, Cai12a, Cai12b, CLZ13, CL13}). There are also some related work done on other classes of projective varieties such as cyclic covers of the projective spaces and complete intersections in projective spaces (\cite{JL15, CPY5, Pan15, LP17}).

On surfaces of general type there are no numerically trivial automorphisms if the irregularity of the surface is $q\geq 3$; on the other hand, if the irregularity is two or less then there indeed exist unbounded series of surfaces, all isogenous to a product of curves, which have numerically trivial automorphisms (\cite{CLZ13, CL13}). We prove in this paper that, if the irregularity is two, these examples are the only ones:
\begin{thm}\label{thm: main}
Let $S$ be a minimal smooth projective surface of general type with $q(S)=2$. If $S$ has a non-trivial automorphism that is numerically trivial, then $S$ is a surface isogenous to a product, of unmixed type.
\end{thm}
Recall that a surface $S$ of general type is isogenous to a product of curves if it admits a product of two smooth curves, say $C\times D$, as an \'etale cover. Indeed, one can assume that the covering $C\times D\rightarrow S$ is Galois, and $S$ is said to be of \emph{unmixed type} if the Galois group does not interchange the two factors of $C\times D$. For the basic properties of such surfaces we refer to the seminal paper \cite{Cat00}, see also \cite[Section~4]{CLZ13}. Theorem \ref{thm: main} is parallel to a result of \cite{CL13}, which says that minimal surfaces of general type with $q=1$ and with a maximal possible automorphism group (of order 4) acting trivially on cohomology are isogenous to a product of two curves.

A few words about the proof of Theorem~\ref{thm: main}. Using the characterization of the surfaces of maximal Albanese dimension on the Severi line (\cite{LY5, BPS15}), we manage to prove that the Albnanese map $a_S\colon S\rightarrow \alb(S)$ is a (flat) bidouble cover, i.e., a Galois branched covering with Galois group $(\Z/2\Z)^2$ (see \cite{Cat84}). The components of the branch curve of $a_S$ are elliptic curves, giving an isogeny of $\alb(S)$ with a product of elliptic curves. This in turn induces a fibration on $S$ whose singular fibres are of the form $2C$ with $C$ smooth. The fibration structure together with the numerical equality $K_S^2=8\chi(\mo_S)$, obtained in \cite{CLZ13}, is enough to conclude that the surface $S$ is isogenous to a product (see \cite{Ser95}).

A surface isogenous to a product has a K\"ahler metric with nonpositive curvature, induced from the product of curves covering it. Using Theorems~\ref{thm: main} and \ref{thm: harmonic}, one proves
\begin{cor}\label{cor: rigid}
Let $S$ be a surface of general type with $q(S)=2$. Then $S$ has no nontrivial automorphism that is homotopic to the identity. In particular, $S$ is rigidified, that is, there is no nontrivial automorphism of  $S$ lying in the identity component $\diff^0(S)$ of the diffeomorphism group.
\end{cor}
Fabrizio Catanese (\cite{Cat13, Cat15}) asked if smooth projective varieties of general type are rigidified. The above corollary answers his question in the positive in the case of surfaces of general type with $q(S)=2$. In general, a positive answer would be useful in establishing a desired local homeomorphism between the Teichm\"uller space and the Kuranishi space at the given complex structure of the manifold.

\medskip

\noindent{\bf Notation and Conventions.} We work over the complex numbers $\C$.

Let $Y$ be a smooth projective variety of dimension $n$. Then 
\begin{itemize}[leftmargin=*]
 \item for a sheaf $\mf$ on $Y$, $h^i(Y,\mf)$ is the dimension of its $i$-th cohomology group $H^i(Y,\mf)$ and $\chi(\mf)$ the Euler characteristic;
 \item $q(Y):=h^1(Y,\mo_Y)$ and $p_g(Y):=h^0(Y, K_Y)$ are the irregularity and the geometric genus of $Y$ respectively;
 \item $e(Y)$ is the topological Euler characteristic;
 \item the Albanese variety of $Y$ is denoted by $\alb(Y)$ and the Albanese map by $a_Y\colon Y\rightarrow\alb(Y)$;
 \item  the full group of biholomorphic automorphisms will be denoted by $\aut(Y)$ and the group of automorphisms acting trivially on the cohomology ring $H^*(Y,\C)$ will be denoted by $\aut_0(Y)$. 
 \end{itemize}

For a finite group $G$ we will denote its order by $|G|$. If it acts on a set $Y$ then $\fix(\sigma):=\{y\in Y\, \mid\, \sigma(y)=y\}$ denotes the fixed point set of an element $\sigma\in G$.

\section{Preliminaries}
\label{sec: pre}
Let $Y$ be a smooth projective variety and $G\subset\aut(Y)$ a finite group of automorphisms inducing trivial action on the cohomology. We recall several basic properties concerning the quotient map $\pi\colon Y\rightarrow Y/G$ (\cite[Section~1]{CL13}).

\begin{lem}\label{lem: basic}
Let $Y$ be a smooth projective variety and $G$ a finite group of automorphisms acting trivially on $H^*(Y,\C)$. \begin{enumerate}
\item Let $X\rightarrow Y/G$ be a resolution of singularities. Then $h^i(X, \mo_X) = h^i(Y,\mo_Y)$ for any $0\leq i \leq \dim Y$. As a consequence, $$ q(X)=q(Y),\, p_g(X) =p_g(Y) \text{ and }\chi(\mo_X)=\chi(\mo_Y).$$
\item If the topological Euler characteristic $e(Y)\neq 0$, then the Albanese map of $Y$ factors as
 \[
a_Y\colon Y\xrightarrow{\pi} Y/G\rightarrow \alb(Y)
  \]
where $\pi\colon Y\rightarrow Y/G$ is the quotient map.
\end{enumerate}
\end{lem}

By the universality of the Albanese maps and Lemma~\ref{lem: basic} (ii) we know that the Albanese varieties $\alb(X)$ and $\alb(Y)$ can be identified after fixing suitable base points for the Albanese maps. Indeed, we have a commutative diagram 
\begin{equation}\label{diag: alb1}
 \begin{aligned}
 \xymatrix{
   Y\ar[r]^\pi & Y/G\ar[r]\ar@{=}[d] &\alb(Y)\ar@{=}[d]\\
  X  \ar[r] & Y/G\ar[r] &\alb(X).
   }
    \end{aligned}
   \end{equation}
\noindent{\bf Convention.} The two identified Albanese varieties $\alb(X)$ and $\alb(Y)$ will be denoted by $A$ if no confusion arises.

We need a result from geometric analysis, namely, the uniqueness of harmonic maps into Riemannian manifolds with nonpositive sectional curvature. 
\begin{thm}\label{thm: harmonic}
Let $(M, g)$ and $(N,h)$ be compact Riemannian manifolds, where $g$ and $h$ denote the Riemannian metrics. Suppose that the sectional curvature of $(N, h)$ is nonpositive. If $\phi_0$ and $\phi_1$ are homotopic harmonic maps from $(M,g)$ to $(N,h)$ such that $\phi_0(p)=\phi_1(p)$ for some $p\in M$, then $\phi_0=\phi_1$.
\end{thm}
\begin{proof}
This is a direct consequence of \cite[(G)]{Har67}, see also \cite[(5.5)]{EL78}.
\end{proof}
We refer to \cite{ES64} and \cite{EL78} for the notions appearing in Theorem~\ref{thm: harmonic}. For maps between K\"ahler manifolds one has the implications holomorphic $\Rightarrow$  harmonic $\Rightarrow$ real-analytic \cite[pages 117--118]{ES64}. 

\section{The Albanese maps}
We begin by recalling the following fact.
 \begin{thm}[\cite{CLZ13}]\label{thm: q geq 2}
Let $S$ be a minimal smooth projective surface of general type with $q(S)=2$ such that $\aut_0(S)$ is nontrivial. Then
\begin{enumerate}
\item $S$ has maximal Albanese dimension;
\item $K_S^2= 8\chi(\mathcal O_S)$;
\item $\aut_0(S)$ has order $2$, say, generated by $\sigma$; the fixed locus $\fix(\sigma)$ consists of exactly $4\chi(\mathcal O_S)$ points.
\end{enumerate}
\end{thm}

\begin{prop}\label{prop: X}
Let $S$ be a minimal smooth projective surface of general type with $q(S)=2$. Assume that $\aut_0(S)$ is nontrivial, and let $\sigma$ be its generating involution. Let $X$ be the minimal resolution of $S/\sigma$. Then the following holds.
\begin{enumerate}
\item $X$ is a minimal surface of general type with $\chi(\mathcal O_X) = \chi(\mathcal O_S)$, $q(X)=2$ and $K_X^2=4\chi(\mo_X)$. 
\item  $X$ is of maximal Albanese dimension.
\end{enumerate}
\end{prop}
\begin{proof}
(i) By Theorem~\ref{thm: q geq 2} (iii) the quotient surface $S/\sigma$ has exactly $4\chi(\mathcal O_S)$ singularities, all of which are ordinary nodes. Let $\pi: S\rightarrow S/\sigma$ be the quotient map.  Then we have $\lambda^*K_{S/\sigma} = K_S$ which is big and nef. It follows that the minimal resolution $X$ of singularities of $S/\sigma$ is a minimal surface of general type, and
\begin{equation}\label{eq: K^2}
K_S^2= 2 K_{S/\sigma}^2 = 2 K_X^2.
\end{equation}
By Lemma~\ref{lem: basic} (i) we have $\chi(\mathcal O_X) = \chi(\mathcal O_S)$ and $q(X)=q(S)=2$. One infers by Theorem~\ref{thm: q geq 2} (ii) that $K_X^2=4\chi(\mathcal O_X)$.

(ii) The surjectivity of the Albanese map $a_X\colon X\rightarrow \alb(X)$ follows from Theorem~\ref{thm: q geq 2} (i) and the diagram \eqref{diag: alb1} with $Y=S$.
\end{proof}

\medskip

In the following the two identified Albanese varieties $\alb(X)$ and $\alb(S)$ will be denoted by $A$, see the convention introduced in Section~\ref{sec: pre}.

\begin{prop}\label{prop: severi}
Let $S$ be a minimal smooth projective surface of general type with $q(S)=2$. Assume that $\aut_0(S)$ is nontrivial, and let $\sigma$ be its generating involution. Then the morphism $a_{S/\sigma}\colon S/\sigma\rightarrow A$, induced by the Albanese map of $S$, is a flat double cover branched along a simple normal crossing ample curve $D$ whose irreducible components are elliptic curves. 
\end{prop}	
\begin{proof}

By Proposition~\ref{prop: X}, the invariants of $X$ lie on the Severi line $K^2=4\chi$. By \cite{BPS15, LY5} we infer that the Albanese map $a_X\colon X\rightarrow A$ is a generically finite map of degreee 2, and the branch locus is an ample curve, say $D$, with at most simple singularities. The \emph{flat} double cover of $A$ branched along $D$ is exactly the canonical model $X_\can$ of $X$, obtained by contracting all the \emph{$(-2)$}-curves on $X$. Thus it remains to show that $D$ is simple normal crossing and $S/\sigma=X_\can$.  We will achieve this by computing the arithmetic genus $p_a(D)$ in two different ways.

Standard computation for double covers yields $K_X^2  = \frac{D^2}{2}$. By the adjunction formula the arithmetic genus of $D$ is
\begin{equation}\label{eq: paB1}
p_a(D) =1 + \frac{D^2}{2} = 1 + K_X^2 = 1+4\chi(\mathcal O_X),
\end{equation}
where the last equality follows from Proposition~\ref{prop: X} (i).

We explain now another way to compute $p_a(D)$. Since $D$ has only simple singularities, the contraction of $(-2)$-curves $\mu\colon X\rightarrow X_\can$ is the canonical resolution of singularities of $X_\can$ as a double cover of $A$ (cf.~\cite[III.7]{BHPV}): we have the following commutative diagram
\begin{equation*}
\xymatrix{
X \ar[r]^\mu \ar[d]_{\tilde a} & X_\can \ar[d]_a\\
\tilde A\ar[r]^\rho & A
}
\end{equation*}
where $\rho\colon\tilde A\rightarrow A$ is the composition of blow-ups resolving successively the singularities of the branch curve and $\tilde a\colon X\rightarrow \tilde A$ is a double cover branched along the strict transform of $D$ possibly plus some ($-2$)-curves over the triple points of $D$. Remark that the $(-2)$-curves on $X$ are exactly the inverse images of the exceptional curves of $\rho$.

Let $\tilde D\subset \tilde A$ be the strict transform of $D$. Then $\tilde D$ is smooth and there is a relation between $p_a(D)$ and $p_a(\tilde D)$ (cf.~\cite[Cor.~V.3.7]{Har77}):
\begin{equation}\label{eq: pa}
p_a(D) = p_a(\tilde D) + \sum_{p\in D_\sing} \delta_p 
\end{equation}
where $D_\sing$ denotes the singular locus of $D$ and  $\delta_p$ is a positive integer, depending only on the type of the curve singularity $p\in D$.

\begin{defn}
A collection of distinct curves $E_1, \dots, E_k$ on a smooth projective surface is called \emph{even} if the sum $\sum_{1\leq i \leq k} E_i$ is linearly equivalent to $2L$ for some integral divisor $L$. 
\end{defn}

\begin{lem}\label{lem: delta_p}
Let $D$ be the branch curve of the Albanese map $a_X\colon X\rightarrow A$ and $p$ a singularity of $D$. Let $E_1, \dots, E_k\subset X$ be an even collection of \emph{disjoint} ($-2$)-curves. Then
\[
 \delta_p \geq \#\{E_i \mid E_i \text{ is contracted to } p\},
\]
and if the equality holds then $p\in D$ is of type $A_{2m+1}$ for some integer $m\geq 0$.
\end{lem} 
\begin{proof}
Let $C$ be the strict transform of $D$ on the blow-up of $A$ at $p$. Then, according to the type of $p\in D$, one can determine the types of singularities of the curve $C$ over $p$ and in turn the value $\delta_p$ as in the following table (cf.~\cite[Sec.~II.8]{BHPV} and \cite[Cor.~V.3.7]{Har77}):

{\renewcommand{\arraystretch}{1.2} 
\begin{center}
 \begin{tabular}{|c|c|c|c|c|c|c|}\hline
$p\in D$&  $A_n, n\geq 1$ & $D_n, n\geq 4$ & $E_6$ & $E_7$ & $E_8$\\ \hline
$q\in C$ over $p$ & $A_{n-2}$ & $A_{n-5}$ & $A_0$ & $A_1$ & $A_2$\\ \hline
$\delta_p(D)$  & $\lfloor \frac{n+1}{2}\rfloor$ & $1+\lfloor \frac{n}{2}\rfloor$ & 3 & $4$ & $4$\\ \hline
\end{tabular}
\end{center}
where points of type $A_{-1}$ and $A_0$ are meant to be smooth points. 
}

If  $p\in D$ is of type $A_{2m}$  or $E_n$ then there is no non-empty collection of disjoint ($-2$)-curves over $p$ whose sum has even intersection with each component of the exceptional locus $a_X^{-1}(p)$. It follows that in these cases $\#\{E_i \mid E_i \text{ is contracted to } p\}=0$.

If  $p\in D$ is of type $D_n$ then the only non-empty collection of disjoint ($-2$)-curves over $p$, whose sum has even intersection with each component of  $a_X^{-1}(p)$, consists of two end components. It follows that in this case $\#\{E_i \mid E_i \text{ is contracted to } p\}\leq 2$.

The lemma follows by comparing with the corresponding values of $\delta_p$ in the above table. 
\end{proof}

Let $\tilde S\rightarrow S$ be the blow-up at $\fix(\sigma)$. Then the induced morphism $\tilde S\rightarrow X$ is a double cover branched exactly along the exceptional curves $E_1,\dots,E_{4\chi}$ in $X$ over the singular points of $S/\sigma$, where $\chi=\chi(\mo_S)$. So they form an even collection of disjoint ($-2$)-curves, and by Lemma~\ref{lem: delta_p} we have 
\begin{equation}\label{eq: delta}
 \sum_{p\in D_\sing} \delta_p \geq\sum_{p\in D_\sing}\#\{E_i \mid E_i \text{ is contracted to } p\}= 4\chi(\mo_S),
\end{equation}
with equality if and only if  $\delta_p = \#\{E_i \mid E_i \text{ is contracted to } p\}$ for any $p\in D_\sing$.

Write $\tilde D=\cup_{1\leq i\leq k} \tilde D_i$ as the union of (smooth) irreducible components.  Since $\tilde D_i$ has a non-constant morphism to the abelian surface $A$ we infer that $g(\tilde D_i)\geq 1$. Combining \eqref{eq: pa} with \eqref{eq: delta} we can bound from below the arithmetic genus of $D$ as follows:
\begin{equation}\label{eq: paB2}
\begin{aligned}
p_a(D) &= p_a(\tilde D) + \sum_{p\in D_\sing}\delta_p\\
          &= -k + 1  + \sum_{1\leq i\leq k} g(\tilde D_i)+\sum_{p\in D_\sing}\delta_p\\
          &\geq -k + 1 + 4\chi(\mathcal O_S) + \sum_{1\leq i\leq k} g(\tilde D_i)\\
          &\geq 1 + 4\chi(\mathcal O_S) \hspace{2cm} \text{(since $g(\tilde D_i)\geq 1$)}.
\end{aligned}
\end{equation}

In view of \eqref{eq: paB1} the inequalities in \eqref{eq: paB2} are both equalities:
\begin{align}
&\delta_p = \#\{E_i \mid E_i \text{ is contracted to } p\}\text{ for any }p\in D_\sing,\label{Dsing}\\ 
&g(\tilde D_i) = 1 \text{ for all } 1\leq i\leq k.\label{eq: gDi}
\end{align}
By \eqref{Dsing} and Lemma~\ref{lem: pi1} the branch curve $D$ has at most $A_{2m+1}$-singularities. The irreducible components $D_i=\rho(\tilde D_i)$ has geometric genus 1 by \eqref{eq: gDi}. Since there are no singular elliptic curves on an abelian variety, the components $D_i$ are in fact smooth. Moreover, the singularities of a union of elliptic curves on an abelian surface are ordinary, hence $D$ has only $A_1$-singularities.

It is now easy to see that the only $(-2)$-curves on $X$ are the ones over the singular points of $S/\sigma$, and the quotient surface $S/\sigma$ is the canonical model of $X$.
\end{proof}

\begin{cor}\label{cor: finite}
Let $S$ be a minimal smooth projective surface of general type with $q(S)=2$. Assume that $\aut_0(S)$ is nontrivial. Then the Albanese map $a_S\colon S\rightarrow A$ is a finite morphism of degree $4$.
\end{cor}
\begin{proof}
Let $\sigma$ be the generating involution of $\aut_0(S)$.  Then the Albanese map $a_S$ is the composition of the quotient map $\pi\colon S\rightarrow S/\sigma$ and the induced map $a_{S/\sigma}\colon S/\sigma\rightarrow A$, both of which are finite of degree 2.
\end{proof}

Let $S$ be a minimal smooth projective surface of general type with $q(S)=2$. Assume that $\aut_0(S)$ is nontrivial, generated by an involution $\sigma$. The flat double cover $a_{S/\sigma}\colon S/\sigma\rightarrow A$ induces an involution $\bar\tau\colon S/\sigma\rightarrow S/\sigma$. The fixed locus $\fix(\bar\tau)$ is the ramfication curve of the double cover and contains all the singularities of $S/\sigma$. We want to lift $\bar\tau$ to $S$, so that the degree 4 finite morphism $a_S\colon S\rightarrow A$ will be recognized as a (flat) bidouble cover (see Proposition~\ref{prop: bidouble} below).

Some preparation is needed. Let $U_0$ be the smooth locus of $S/\sigma$, which is invariant under the action of $\bar\tau$.
\begin{lem}\label{lem: pi1}
Let $u\in U_0$ be a $\bar\tau$-fixed point. Then the induced automorphism of the fundamental group $\bar\tau_*\colon \pi_1(U_0, u)\rightarrow \pi_1(U_0,u)$ is the identity map.
\end{lem}
\begin{proof}
Let $p_1,\dots,p_{4\chi} \in S$ be the fixed points of $\sigma$, where $\chi=\chi(\mathcal O_S)$. Then their images $q_i$ in $S/\sigma$ are exactly the singular points of $S/\sigma$, so we have $U_0=(S/\sigma) \setminus\{q_1,\dots,q_{4\chi}\}$. The images $a_S(p_i)$ in the Albanese surface $A$ are exactly the nodes of the branch curve $D$.

For $1\leq i\leq 4\chi$ we take a $\bar\tau$-invariant open neighborhoods $U_i\subset S/\sigma$ of $q_i$ in the Euclidean topology such that $U_i$ is analytically isomorphic to $(x^2=yz)\subset\B^3$, where $\B^3$ is the unit ball in $\C^3$, and the action of $\bar\tau$ is given by $(x,y,z)\mapsto(-x,y,z)$. Then $U_i$ is simply connected due to the conic structure at the singularities (cf.~\cite[page 23]{D92}). Moreover, one can assume that the $U_i$'s are so small that they are pairwise disjoint.

We view $S/\sigma$ as the topological space obtained by patching the small neighborhoods $U_i$'s to $U_0$. More precisely, set $X_0 = U_0$ and define $X_i=X_{i-1}\cup U_i$ inductively for $1\leq i\leq 4\chi$. Then it is clear that $S/\sigma = X_{4\chi}$. For each $1\leq i\leq 4\chi$ let $u_i\in a_{S/\sigma}^{-1}(D)\cap(U_i\setminus\{q_i\})$ be a $\bar\tau$-invariant point. Then there is a exact sequence of fundamental groups by van Kampen's theorem
\begin{equation}\label{eq: pi1}
 \pi_1(U_i\setminus\{q_i\}, u_i) \rightarrow \pi_1(X_{i-1}, u_i) \rightarrow \pi_1(X_i, u_i) \rightarrow 1
\end{equation}
which is preserved by induced involution $\bar\tau_*$ on the fundamental groups.

We prove by reversed induction on $i$ that $\bar\tau_*\colon \pi_1(X_{i-1})\rightarrow\pi_1(X_{i-1})$ is the identity map. Here we omit the base points for the fundamental groups to simplify the notation because, for the statement to hold, the base points are irrelevant. By a result of Nori (\cite[Corollary~2.7]{Nor83}) there is an isomorphism 
\[
\pi_1(X_{4\chi})=\pi_1(S/\sigma)\cong \pi_1(A),
\]
Therefore, as the base step of induction, $\bar\tau_*$ acts trivially on $\pi_1(X_{4\chi})$.  Concerning the left end of \eqref{eq: pi1} we have
\begin{equation}
 \pi_1(U_i\setminus\{q_i\}) \cong \Z/2\Z,
\end{equation}
which automorphism group is trivial. In particular, $\bar\tau_*$ acts as identity on it.  It follows that, if $\bar\tau_*$ is the identity on $\pi_1(X_i)$, so is it  on  $\pi_1(X_{i-1})$. This finishes the induction step and we conclude that $\bar\tau_*$ is the identity on $\pi_1(U_0)$.
\end{proof}

\begin{prop}\label{prop: bidouble}
Let $S$ be a minimal smooth projective surface of general type with $q(S)=2$. Assume that $\aut_0(S)$ is nontrivial. Then the Albanese map $a_S\colon S\rightarrow A$ is a bidouble cover.
\end{prop}
\begin{proof}
As before, let $\sigma$ be the generating involution of $\aut_0(S)$. We retain the notation in Lemma~\ref{lem: pi1}. Let $S_0\subset S$ be the inverse image of $U_0$ under the quotient map $\lambda\colon S\rightarrow S/\sigma$. Then the map $\lambda\restr{S_0}\colon S_0\rightarrow U_0$ is an \'etale double cover.  By Lemma~\ref{lem: pi1} the induced automorphism $\bar\tau_*$ of $\pi_1(U_0,u)$ is the identity, where $u\in U_0$ is a $\bar\tau$-fixed point. In particular, the subgroup $(\lambda\restr{S_0})_*\pi_1(S_0,s)$ is invariant by $\bar\tau_*$, where $s\in S_0$ is chosen to be over $u$. As is known from general topology there is an automorphism $\tau_0$ of $S_0$ such that the following diagram commutes
 \[
\xymatrix{
S_0 \ar[r]^{\tau_0} \ar[d]& S_0 \ar[d]\\
U_0 \ar[r]^{\bar\tau\restr{U_0}} & U_0.
}
\]
By the Riemann extension theorem $\tau_0$ extends to an automorphism $\tau$ of $S$, which is necessarily a lifting of $\bar\tau$.

The group generated by $\sigma$ and $\tau$ sits in an extension of  an order 2 group by the other:
 \[
1\rightarrow \langle\sigma\rangle\rightarrow \langle\sigma,\tau\rangle\rightarrow \langle\bar\tau\rangle\rightarrow 1,
\]
hence is an abelian group of order 4.

Now we have a factorization of the Albanese map of $S$:
\[
a_S\colon S\rightarrow S/\langle\sigma,\tau\rangle \rightarrow A.
 \]
Since $\deg(a_S) =|\langle\sigma,\tau\rangle|= 4$, the finite morphism between normal surfaces $S/\langle\sigma,\tau\rangle \rightarrow A$ is birational, hence is an isomorpism.

We  claim that $\langle\sigma,\tau\rangle\cong(\Z/2\Z)^2$, hence the finite morphism $a_S\colon S\rightarrow A$ is a bidouble cover.  Otherwise, $\langle\sigma,\tau\rangle$ is isomorphic to $\Z/4\Z$. The two automorphisms $\tau$ and $\tau\circ\sigma$ must be of order 4 and we have $\sigma=\tau^2=(\sigma\circ\tau)^2$. Hence the fixed point sets $\fix(\tau)$ and $\fix(\tau\circ\sigma)$ are both contained in $\fix(\sigma)$. But the latter consists only of isolated points by Theorem~\ref{thm: q geq 2}, contradicting the fact that $a_S\colon S\rightarrow A$ has a non-empty branch curve $D$. \end{proof}

\section{Proofs of the main results}
\begin{proof}[Proof of Theorem~\ref{thm: main}]
Let $\sigma_0=\sigma,\, \sigma_1,\,\sigma_2$ be the three nontrivial elements of the Galois group of the bidouble cover $a_S\colon S\rightarrow A$. Note that $\sigma_0$ does not fix any curve. For $i=1,2$ let $D_i$ be the branch curve, the stabilizer over which is generated by $\sigma_i$. Since $S$ is smooth, the branch curves $D_1$ and $D_2$ are smooth (cf.~\cite{Cat84, Cat99}). So they are both disjoint union of smooth elliptic curves (cf.~Proposition~\ref{prop: severi}).

Now it is easy to see that $D_i$ consists of fibres of some smooth elliptic fibration $h_i\colon A\rightarrow E_i$. Composing $h_1$ with the Albanese map $a_S\colon S\rightarrow A$ we get a fibration $f\colon S\rightarrow E_1$. (The fibration $h_2\circ a_S\colon S\rightarrow E_2$ also works.) One sees that the singular fibres of $f$ are over $D_1$ and they are of the form $2C$ with $C$ smooth. With such a fibration structure and with the numerical equality $K_S^2=8\chi(\mathcal O_S)$ (see Theorem~\ref{thm: q geq 2}) the surface $S$ must be isogenous to a product, of unmixed type (\cite[Lemma~5]{Ser95}).
\end{proof}

\begin{rmk}
Surfaces isogenous to a product with $q=2$ and with nontrivial $\aut_0(S)$ have been classified in \cite[Theorem~4.9]{CLZ13}.
\end{rmk}

\begin{proof}[Proof of Corollary~\ref{cor: rigid}]
Since automorphisms homotopic to the identity are numerically trivial, the assertion of the corollary is clear if $\aut_0(S)$ is trivial. 

Now assume that $\aut_0(S)$ is not trivial and $\id_S\neq \sigma\in\aut_0(S)$. By the Lefschetz fixed point theorem, $e(\fix(\sigma))=e(S)>0$. In particular, $\fix(\sigma)$ is not empty. Let $\rho\colon S\rightarrow S_{\min}$ be the contraction to its smooth minimal model. Then $\rho\circ\sigma$ and $\rho$ are distinct morphisms and coincide on $\fix(\sigma)$.

By the uniqueness of the minimal model, the automorphism $\sigma$ descends to a nontrivial automorphism $\sigma_{\min}$ of $S_{\min}$ such that $\rho\circ\sigma=\sigma_{\min}\circ\rho$. The fact that $\sigma$ is numerically trivial implies that $\sigma_{\min}$ is also numerically trivial. By Theorem~\ref{thm: main}, the surface $S_{\min}$ is isogenous to a product of curves. In particular, one sees that $S_{\min}$ has a K\"ahler metric with nonpositive sectional curvature. Since holomorphic maps between K\"ahler manifolds are harmonic with respect to the given K\"ahler metrics, we infer that $\rho\circ\sigma$ is not homotopic to $\rho$ by Theorem~\ref{thm: harmonic}, taking into account that the two maps coincide on the non-empty set $\fix(\sigma)$. It follows that $\sigma$ is not homotopic to $\id_S$.
\end{proof}

\begin{rmk}
The universal cover of a surface of general type isogenous to a product is a bidisk, which is a bounded domain. One could have applied the uniqueness result of Borel and Narasimhan \cite[Theorem~3.6]{BN67} to give another proof of Corollary~\ref{cor: rigid},  provided that automorphisms homotopic to the identity induce the trivial action on the fundamental group with a given base point. This is however not clear. In general, a base point preserving homeomorphism homotopic to the identity induces only an inner automorphism of the fundamental group. Due to this observation, the proof of \cite[Proposition~4.8]{CLZ13} is incomplete, but its statement is still valid by the argument for Corollary~\ref{cor: rigid} given here. Likewise, in the first paragraph of \cite{BN67}, "continuously homotopic" seems not enough -- the homotopy should preserve a base point.
\end{rmk}

\end{document}